\documentclass[10pt]{amsart}

\usepackage{color,graphicx,shortvrb}
\usepackage[latin 1]{inputenc}

\usepackage[bookmarksnumbered, colorlinks, plainpages]{hyperref}
\hypersetup{colorlinks=true,linkcolor=red, anchorcolor=green, citecolor=cyan, urlcolor=red, filecolor=magenta, pdftoolbar=true}

\usepackage{enumerate}
\usepackage{amssymb}
\usepackage{tikz-cd}

\usepackage [ all ]{xy}

\newtheorem{theorem}{Theorem}[section]

\newtheorem{corollary}[theorem]{Corollary}

\newtheorem{lemma}[theorem]{Lemma}

\newtheorem{proposition}[theorem]{Proposition}
\newtheorem{remark}[theorem]{Remark}

\def\J#1#2#3{ \left\{ #1,#2,#3 \right\} }

\def\11{\textbf{$1$}}
\def\CC{{\mathbb{C}}}

\DeclareGraphicsExtensions{.jpg,.pdf,.png,.eps}

\begin{document}

\title[On the extension problem]{On the extension of surjective isometries whose domain is the unit sphere of a space of compact operators}

\author[A.M. Peralta]{Antonio M. Peralta}

\address{Departamento de An{\'a}lisis Matem{\'a}tico, Facultad de
Ciencias, Universidad de Gra\-na\-da, 18071 Granada, Spain.}
\email{aperalta@ugr.es}




\date{}

\begin{abstract}
We prove that every surjective isometry from the unit sphere of the space $K(H),$ of all compact operators on an arbitrary complex Hilbert space $H$, onto the unit sphere of an arbitrary real Banach space $Y$ can be extended to a surjective real linear isometry from $K(H)$ onto $Y$. This is probably the first example of an infinite dimensional non-commutative C$^*$-algebra containing no unitaries and satisfying the Mazur--Ulam property. We also prove that all compact C$^*$-algebras and all weakly compact JB$^*$-triples satisfy the Mazur--Ulam property.
\end{abstract}

\maketitle
\thispagestyle{empty}

\section{Introduction}

The problem of extending surjective isometries between the unit spheres of two Banach spaces has experienced a substantial turn with the introduction, by M. Mori and N. Ozawa, of the so-called strong Mankiewicz property. The celebrated Mazur--Ulam theorem has been a source of inspiration for many subsequent research. A key piece among the different generalizations that appeared later is due to P. Mankiewicz \cite{Mank1972}. Let us recall that a convex subset $K$ of a normed space $X$ is called a \emph{convex body} if it has non-empty interior in $X$. P. Mankiewicz proved in \cite{Mank1972} that every surjective isometry between convex bodies in two arbitrary normed spaces can be uniquely extended to an affine function between the spaces. M. Mori and N. Ozawa introduced the strong Mankiewicz property in \cite{MoriOza2018}. According to the just quoted paper, a convex subset $K$ of a normed space $X$ satisfies the \emph{strong Mankiewicz property} if every surjective isometry $\Delta$ from $K$ onto an arbitrary convex subset $L$ in a normed space $Y$ is affine. It is established by Mori and Ozawa that for a Banach space $X$ satisfying that the closed convex hull of the extreme points, $\partial_e (\mathcal{B}_X),$ of the closed unit ball, $\mathcal{B}_X$, of $X$ has non-empty interior in $X$, every convex body $K\subset X$ satisfies the strong Mankiewicz property (see \cite[Theorem 2]{MoriOza2018}). By the Russo-Dye theorem every convex body of a unital C$^*$-algebra satisfies the strong Mankiewicz property, and the same property holds for convex bodies in real von Neumann algebras (see \cite[Corollary 3]{MoriOza2018}) and JBW$^*$-triples (cf. \cite[Corollary 2.2]{BeCuFerPe2018}).\smallskip

Based on the strong Mankiewicz property, M. Mori and N. Ozawa proved that unital C$^*$-algebras and real von Neumann algebras are among the spaces satisfying the Mazur--Ulam property, that is, every surjective isometry from the unit sphere, $S(A),$ of a unital C$^*$-algebra, or of a real von Neumann algebra $A$, onto the unit sphere, $S(Y)$, of an arbitrary real Banach space $Y$, admits a unique extension to a surjective real linear isometry between the spaces (see \cite{MoriOza2018}). J. Becerra Guerrero, M. Cueto-Avellaneda, F.J. Fern{\'a}ndez-Polo and the author of this note showed that every JBW$^*$-triple $M$ which is not a Cartan factor of rank two always satisfies the Mazur--Ulam property (cf. \cite{BeCuFerPe2018}). In a recent collaboration with O.F.K. Kalenda we prove that every rank-2 Cartan factor satisfies the Mazur--Ulam property \cite[Theorem 1.1]{KalPe2019}, and consequently, every JBW$^*$-triple enjoys the Mazur--Ulam property \cite[Corollary 1.2]{KalPe2019}. These results simply are the most recent advances of a long list of paper studying the extension of isometries between the unit spheres of two Banach spaces (see, for example, \cite{
 CabSan19, CuePer18, CuePer19, FerJorPer2018, FerPe17c, FerPe17b, FerPe17d, FerPe18Adv, JVMorPeRa2017, Mori2017, MoriOza2018, Pe2019, PeTan16, Ting1987, WH19} and the surveys \cite{YangZhao2014, Pe2018}). \smallskip

So the natural question is: what can we say in the case of Banach spaces or C$^*$-algebras whose closed unit ball contains no extreme points? This is the case of the space $K(H)$, of all compact operators on an infinite dimensional complex Hilbert space $H$, where $\partial_e(\mathcal{B}_{K(H)}) = \emptyset.$ The lacking of extreme points of the closed unit ball makes impossible a straight application of the arguments based on the strong Mankiewicz property.\smallskip

In \cite{PeTan16},  R. Tanaka and the author of this note proved that every surjective isometry between the unit spheres of two compact C$^*$-algebras (in particular between the unit spheres of two spaces of compact linear operators on a complex Hilbert space) extends (uniquely) to a surjective real linear isometry between the two C$^*$-algebras. In collaboration with F.J. Fern{\'a}ndez-Polo we showed that the same conclusion remains valid for a surjective isometry between the unit spheres of two complex Banach spaces in the strictly wider class of weakly compact JB$^*$-triples (cf. \cite{FerPe18Adv}). The reader can get access to the concrete definitions of these objects in the subsequent subsection \ref{subsec:background}.
In this paper we prove that weakly compact JB$^*$-triples satisfy a much stronger property, namely, every weakly compact JB$^*$-triple $E$ satisfies the Mazur--Ulam property, in equivalent words, every surjective isometry from the unit sphere of $E$ onto the unit sphere of an arbitrary Banach space $Y$ extends to a surjective real linear isometry from $E$ onto $Y$ (cf. Theorem \ref{t weakly compact JB*-triples satisfy MUP}).\smallskip

Our strategy to solve the problem determines the structure of this note. In Section \ref{sec:new facial properties} we gather some new results derived from our knowledge on the facial structure of elementary JB$^*$-triples. New technical properties of elementary JB$^*$-triples, established in Propositions \ref{p new geometric facial property} and \ref{p first consequence from the new geometric facial property}, are applied to deduce that every surjective isometry $\Delta: S(K) \to S(Y),$ where $K$ is an elementary JB$^*$-triple and $Y$ is a real Banach space, is affine on every non-empty convex subset $\mathcal{C}\subset S(K)$ (cf.  Corollary \ref{c affine on convex subsets of the sphere in reflexive}).\smallskip

Section \ref{sec: elementary JB* triples satisfy MUP} is aimed to prove that every elementary JB$^*$-triple satisfies the Mazur--Ulam property (see Theorem \ref{t elementary JBstar triples satisfy the MUP}). In particular, for any complex Hilbert space $H$, the space $K(H)$ satisfies the Mazur--Ulam property (cf. Corollary \ref{c KH satisfies the MUP}). This closes a natural question which remained open until now. Let us observe that in case that $H$ is infinite dimensional the closed unit ball of $K(H)$ lacks of extreme points. As shown in \cite{JVMorPeRa2017}, $c_0$ satisfies the Mazur--Ulam property. Probably, the results in this note show the first example of a non-commutative non-unital C$^*$-algebra satisfying this property.\smallskip

As we know from many other mathematical problems, certain questions are easier to answer from a more general point of view. This is the case of the Mazur--Ulam property for the space of compact operators; the arguments in the wider class of elementary and weakly compact JB$^*$-triples are probably more abstract but offer an accesible proof.\smallskip

Our goal in the second part of the paper is to prove that every weakly compact JB$^*$-triple satisfies the Mazur--Ulam property (see Theorem \ref{t weakly compact JB*-triples satisfy MUP}). For this purpose, we shall first show that the closed unit ball of any weakly compact JB$^*$-triple enjoys the strong Mankiewicz property (cf. Corollary \ref{c weakly compact JB* triples satisfy the SMP}). A consequence of this second main result, which is worth to be stated by its own importance, shows that every compact C$^*$-algebra (that is, a C$^*$-algebra which coincides with a $c_0$-sum of spaces of compact operators) also has the Mazur--Ulam property (cf. Corollary \ref{c compact C*-algebras satisfy MUP final}).

\section{Basic background and references}\label{subsec:background}

The Riemann mapping theorem is one of the best known results in the theory of holomorphic functions of one variable. As it was already observed by H. Poincar{\'e} in 1907, for higher dimensional Banach spaces the conclusion of the Riemann mapping theorem is no longer valid, there are lots of simply connected domains which are not biholomorphic to the open unit ball. \emph{Bounded symmetric domains} in finite dimensions were introduced and
completely classified by E. Cartan. L. Harris proved in 1974 that the open unit ball of every C$^*$-algebra is a bounded symmetric domain \cite{Harris74}, a conclusion which remains true for JB$^*$-algebras (see \cite{BraKaUp78}). The most conclusive study was obtained by W. Kaup who
proved that every bounded symmetric domain in a complex Banach space is biholomorphically equivalent to the open unit ball of a
JB$^*$-triple (cf. \cite{Ka83}).\smallskip

A complex Banach space $E$ is called a \emph{JB$^*$-triple} if it can be equipped with a continuous triple product $\J ... : E\times E\times E \to E,$ which is conjugate linear in the middle variable and symmetric and bilinear in the outer variables satisfying the following axioms:
\begin{enumerate}[{\rm $(a)$}] \item (Jordan Identity) $L(a,b) L(x,y) - L(x,y) L(a,b)= L(L(a,b)x,y) - L(x, L(b,a) y),$ for all $a,b,x,y,$ in $E$, where $L(x,y)$ is the linear operator defined by $L(a,b) (z) =\{a,b,z\}$ ($\forall z\in E$);
\item The operator $L(a,a): E\to E$ is hermitian and has non-negative spectrum;
\item $\|\J aaa\| = \|a\|^3$, for every $a\in E$.\end{enumerate}\smallskip

It can be found in the previously mentioned references that every C$^*$-algebra $A$ is a JB$^*$-triple with respect to the triple product \begin{equation}\label{eq triple product on Cstaralg} (a,b,c)\mapsto \{a,b,c\} =1/2 (a b^* c + c b^* a), \ \ (a,b,c\in A).
 \end{equation} This triple product also induces a structure of JB$^*$-triple for the space $B(H_1,H_2)$ of all bounded linear operators between two complex Hilbert spaces $H_1$ and $H_2$, and for every closed subspace of $B(H_1,H_2)$ which is closed for this triple product. In particular every complex Hilbert space and the space $K(H_1,H_2),$ of all compact linear operators from $H_1$ to $H_2,$ are JB$^*$-triples. The class of JB$^*$-triples is also widen with all JB$^*$-algebras when they are equipped with the triple product given by \begin{equation}\label{eq triple product JB*-algebras} \J abc = (a \circ b^*) \circ c + (c\circ b^*) \circ a - (a\circ c) \circ b^*.
 \end{equation}

A subspace $B$ of a JB$^*$-triple $E$ is a JB$^*$-subtriple of $E$ if $\{B,B,B\}\subseteq B$. A JB$^*$-subtriple $I$ of $E$ is called an \emph{inner ideal of $E$} if $\{I,E,I\}\subseteq I$. A subspace $I$ of a C$^*$-algebra $A$ is called an \emph{inner ideal} if $I A I \subseteq I$. For example, if $p$ and $q$ are projections in a C$^*$-algebra $A$, the subspace $I= p A q$ is an inner ideal of $A$. Inner ideals of JB$^*$-triples are studied and characterized in \cite{EdRutt92}.\smallskip

A \emph{JBW$^*$-triple} is a {JB$^*$-triple} which is also a dual Banach space. Every von Neumann algebra is a JBW$^*$-triple. The theory of JB$^*$-triple runs in parallel to the theory of C$^*$-algebras. For example, the second dual of a JB$^*$-triple $E$ is a JBW$^*$-triple under a triple product extending the product of $E$ \cite{Di86}. It is also known that every JBW$^*$-triple admits a unique isometric predual, and its triple product is separately weak$^*$ continuous \cite{BarTi86}. \smallskip

Let us recall that an element $e$ in a C$^*$-algebra $A$ is called a \emph{partial isometry} if $ee^*$ (equivalently, $e^*e$) is a projection. It is known that $e$ is a partial isometry if and only if $ee^* e =e$. It is easy to see that, in terms of the triple product given in \eqref{eq triple product on Cstaralg}, an element $e\in A$ is a partial isometry if and only if $\{e,e,e\}=e$. In the wider framework of JB$^*$-triples, elements satisfying $\{e,e,e\}$ are called \emph{tripotents}. 
For each tripotent $e\in E$ the eigenvalues of the operator $L(e,e)$ are precisely $0,1/2$ and $1$. For $j\in \{0,1,2\}$, by denoting by $E_j (e)$ the $\frac{j}{2}$-eigenspace of $L(e,e)$, the JB$^*$-triple $E$ decomposes as the direct sum $$E= E_{2} (e) \oplus E_{1} (e) \oplus E_0 (e).$$ This decomposition is called the \emph{Peirce decomposition} of $E$ with respect to the tripotent $e$, and the projection of $E$ onto $E_j(e)$, which is denoted by $P_j(e)$, is called the Peirce $j$-projection (see \cite{Loos2}). It is further known that Peirce projections are contractive (cf. \cite{FriRu85}) and satisfy the following identities $P_{2}(e) = Q(e)^2,$ $P_{1}(e) =2(L(e,e)-Q(e)^2),$ and $P_{0}(e) =Id_E - 2 L(e,e) + Q(e)^2,$ where for each $a\in E$, $Q(a):E\to E$ is the conjugate linear map given by $Q(a) (x) =\{a,x,a\}$. Consequently, in a JBW$^*$-triple Peirce projections are weak$^*$ continuous and Peirce subspaces are weak$^*$ closed.\smallskip

Triple products among Peirce subspaces satisfy certain rules, known as \emph{Peirce rules} or \emph{Peirce arithmetic}, asserting that, for $k,j,l\in \{0,1,2\}$ we have $$\begin{aligned}  \{ E_k(e) , E_j (e), E_l (e)\} &\subseteq E_{k-j+l} (e), \hbox{ if $k-j+l \in\{0,1,2\}$, and } \\
 \{ E_k(e) , E_j (e), E_l (e)\} &= \{0\}\hbox{ otherwise.}
\end{aligned}$$
A tripotent $e$ in $E$ is called \emph{complete} (respectively, \emph{unitary} or \emph{minimal}) if $E_0(E)=\{0\}$  (respectively, $E_2(e)=E$ or $E_2(e)=\CC e \neq \{0\}$).\smallskip

Orthogonality in JB$^*$-triples is another notion required in this note. Elements $a,b$ in a JB$^*$-triple are said to be \emph{orthogonal} (written $a\perp b$) if $L(a,b) =0$. It is known that $a\perp b$ if, and only if, $L(b,a)=0$ if, and only if, $\{a,a,b\} =0$ if, and only if, $\{b,b,a\}=0$ (see \cite[Lemma 1]{BurFerGarMarPe} for more equivalent reformulations). This notion is consistent with the usual concept of orthogonality in C$^*$-algebras. Let $a,b$ be elements in $B(H)$ (or in a general C$^*$-algebra). We say that $a$ and $b$ are \emph{orthogonal} if $ a b^* = b^* a$. It is well known that $\|a\pm b\|= \max\{\|a\|,\|b\|\}$ whenever $a\perp b$ in (see \cite[Lemma 1.3$(a)$]{FriRu85}). For each non-zero finite rank partial isometry $e\in K(H)$ there exists a finite family of mutually orthogonal minimal partial isometries $\{e_1,\ldots, e_m\}$ in $K(H)$ such that $ e = e_1+ \ldots + e_m$.  We say that a tripotent $e$ in a JB$^*$-triple $E$ has \emph{finite rank} if it can be written a sum of finitely many mutually orthogonal minimal tripotents.\smallskip

To make easier our subsequent arguments we remark the following property: let $u$ and $e$ be two orthogonal tripotents in a JB$^*$-triple $E$. Clearly $e+u$ is a tripotent in $E$ and we can easily deduce from Peirce arithmetic that \begin{equation}\label{eq Peirce 2 of an orthogonal sum} E_{2} (e+u) = E_2(e) \oplus E_2(u) \oplus E_1(e) \cap E_1 (u).
\end{equation}

A subset $S$ of a JB$^*$-triple $E$ is called \emph{orthogonal} if $0 \notin S$ and $x \perp y$ for every $x\neq y$ in $S$. The minimal cardinal number $r$ satisfying $card(S) \leq r$ for every orthogonal subset $S \subseteq E$ is called the \emph{rank} of $E$ (cf. \cite{Ka97} and \cite{BeLoPeRo} for basic results on the rank of a Cartan factor and a JB$^*$-triple).\smallskip

Let $B$ be a subset of a JB$^*$-triple $E.$ We shall denote by $B^\perp$ the \emph{(orthogonal) annihilator of $B$} defined by $
B^\perp= B^{\perp}_{_E}:=\{ z \in E : z \perp x , \forall x \in B \}.$ Given a tripotent $e$ in $E$ the inclusions $$ E_2(e) \oplus  E_1(e)\supseteq \{e\}_{_E}^{\perp \perp}= E_0(e)^{\perp} \supseteq E_2(e)$$ always hold (see \cite[Proposition 3.3]{BurGarPe11}). It is also known that the equality $\{e\}_{_E}^{\perp \perp}= E_2(e)$ is not always true. The following counterexample\label{eq counterexample biorthog complement} can be found in \cite[Remark 3.4]{BurGarPe11}: suppose $H_1$ and $H_2$ are two infinite dimensional complex Hilbert spaces and $p$ is a minimal projection in $B(H_1)$. If $E$ denotes the orthogonal sum $p B(H_1) \oplus^{\infty} B(H_2)$ and we consider $e=p$ as an element in $E$, it can be checked that $p$ is a non-complete tripotent in $E$, $\{e\}_{_E}^{\perp} = B(H_2)$ and $\{e\}_{_E}^{\perp\perp} = E_2(e) \oplus E_1(e) =
p B(H_1)\neq \mathbb{C} p = E_2 (e).$\smallskip

Despite of the previous counterexample, if we assume that $E$ is a Cartan factor and $e$ is a non-complete tripotent in $E,$ then the equality $\{e\}^{\perp \perp}=E_0(e)^\perp= E_2(e)$ is always true (see \cite[Lemma 5.6]{Ka97}).\smallskip

This seems an appropriate moment to refresh the definition of Cartan factors. A JB$^*$-triple is called a Cartan factor of type 1 if it is a
JB$^*$-triple of the form $B(H_1, H_2)$, where $H_1$ and $H_2$ are two complex Hilbert spaces. Let $j$ be a conjugation on a complex Hilbert space $H$, and consider the linear involution $x\mapsto x^t:=jx^*j$ on $B(H).$ A Cartan factor of type 2 (respectively,
type 3) is a JB$^*$-triple which coincides with the subtriple of $B(H)$ formed by the $t$-skew-symmetric (respectively, $t$-symmetric) operators. All we need to know in this note about types 4, 5 and 6 Cartan factors is that the first one is reflexive while the last two are finite dimensional (see \cite{Ka97} for more details).\smallskip

According to \cite{BuChu}, given a Cartan factor of type $j\in \{1,\ldots, 6\},$ the elementary JB$^*$-triple $K_j$ of type $j$ is defined in the following terms: $K_1 = K (H_1, H_2)$; $K_i = C \cap K(H)$ when $C$ is of type $i = 2 , 3$, and $K_i = C$ if the latter is of type $ 4, 5,$ or
$6$. Obviously, if $K$ is an elementary JB$^*$-triple of type $j$, its bidual is precisely a Cartan factor of $j$.\smallskip

We establish next a version of \cite[Lemma 5.6]{Ka97} for elementary JB$^*$-triples.

\begin{lemma}\label{l biorthogonal of a non-complete tripotent in an elementary} Let $K$ be an elementary JB$^*$-triple. Suppose $e$ is a non-complete tripotent in $K.$ Then we have $\{e\}^{\perp \perp}=K_0(e)^\perp= K_2(e).$
\end{lemma}

\begin{proof} If $K$ is an elementary JB$^*$-triple of type $j\in \{4,5,6\}$ we know that $K$ is a Cartan factor of the same type and hence the conclusion follows from \cite[Lemma 5.6]{Ka97}.\smallskip

Suppose $K$ is an elementary JB$^*$-triple of type $j\in \{1,2,3\}$. The corresponding bidual $K^{**}= C$ is a Cartan factor of type $j$. By Goldstine's theorem $K$ is weak$^*$ dense in $C$. Since $e\in K$, and $P_0(e)$ is weak$^*$ continuous in $C$, we deduce that $K_0(e) =\{e\}_{_K}^{\perp}$ and $K_2(e)$ are weak$^*$ dense in $C_0(e) =\{e\}_{_C}^{\perp}$ and $C_2(e)$, respectively. Thus the desired conclusion also follows from \cite[Lemma 5.6]{Ka97}. Namely, we know that $\{e\}_{_K}^{\perp \perp}= K_0(e)^{\perp} \supseteq K_2(e)$ \cite[Proposition 3.3]{BurGarPe11}. Let us take $x\in \{e\}_{_K}^{\perp\perp}$. In this case, $\{x,x,z\} =0$ for all $z\in \{e\}_{_K}^{\perp}= K_0(e)$. It follows from the weak$^*$ density of $K_0(e)$ in $C_0(E)$ and the separate weak$^*$ continuity of the triple product of $C$ that $\{x,x,a\}=0$ for all $a\in  \{e\}_{_C}^{\perp}= C_0(e)$. By \cite[Lemma 5.6]{Ka97} we have $x\in \{e\}_{_C}^{\perp \perp}=C_0(e)^\perp= C_2(e),$ and consequently $x\in K_2(e) = C_2(e) \cap K$.
\end{proof}

According to \cite{BuChu}, an element $x$ in a JB$^*$-triple $E$ is called \emph{weakly compact} (respectively, \emph{compact}) if
the operator $Q(x):E\rightarrow E$ is weakly compact (respectively, compact). We say that $E$ is \emph{weakly compact}
(respectively, \emph{compact}) if every element in $E$ is weakly compact (respectively, compact). If we denote by $K(E)$ the Banach
subspace of $E$ generated by its minimal tripotents, then $K(E)$ is a (norm closed) triple ideal of $E$ and it coincides with the
set of weakly compact elements of $E$ (see Proposition 4.7 in \cite{BuChu}).
It follows from \cite[Lemma 3.3 and Theorem 3.4]{BuChu} that a JB$^*$ triple, $E,$ is weakly compact if
and only if one of the following statement holds: \begin{enumerate}[$a)$] \item $K(E^{**})=K(E)$. \item $K(E)=E.$
\item $E$ is a $c_0$-sum of elementary JB$^*$-triples.
\end{enumerate}

Obviously each non-zero tripotent in a weakly compact JB$^*$-triple is of finite rank. It was observed in \cite[Corollary 2.5]{PeTan16} that the results in \cite{BuChu} can be applied to deduce that a JB$^*$-triple $E$ is weakly compact if and only if it contains every tripotent of $E^{**}$ which is compact relative to $E$ in the sense of \cite{EdRu96,FerPe06}.\smallskip

It should be remarked here that weakly compact JB$^*$-triples are the JB$^*$-triple analogue of compact C$^*$-algebras in the sense employed in \cite{Alex,Yli}, with the exception that a C$^*$-algebra is compact if, and only if, it is weakly compact (see \cite{Yli}).

\begin{corollary}\label{c biorthogonal in an elementary through the orthogonal minimal tripotents} Let $K$ be an elementary JB$^*$-triple. Suppose $e$ is a non-complete tripotent in $K.$ Let $x$ be an element in $K$ satisfying $x\perp v$ for every minimal tripotent $v\in K$ with $v\perp e$. Then $x\in K_2(e)$.
\end{corollary}

\begin{proof} Let us observe that from Peirce arithmetic the Peirce 0-subspace $K_0(e) \neq \{0\}$ is an inner ideal of $K$. Corollary 3.5 in \cite{BuChu} together with the fact that $K$ is a factor show that $K_0(e)$ must be a weakly compact JB$^*$-triple. By Remark 4.6 in \cite{BuChu} every element in the weakly compact JB$^*$-triple $K_0(e)$ can be approximated in norm by finite positive linear combinations of mutually orthogonal minimal tripotents in $K_0(e)$. Since every minimal tripotent in $K_0(e)$ is a minimal tripotent in $K$ which is orthogonal to $e$, it follows from the hypothesis that $x\perp K_0(e)^{\perp}= \{e\}^{\perp\perp}$. Lemma \ref{l biorthogonal of a non-complete tripotent in an elementary} implies that $x\in K_2(e)$.
\end{proof}

Let us observe that Corollary \ref{c biorthogonal in an elementary through the orthogonal minimal tripotents} can be also proved by a direct argument in the case in which $K=K(H)$ where $H$ is a complex Hilbert space.\smallskip

We are actually interested in finding conditions on a tripotent $e$ in a weakly compact JB$^*$-triple $E$ to guarantee that the equality $\{e\}^{\perp \perp}=K_0(e)^\perp= K_2(e)$ holds (compare Lemma \ref{l biorthogonal of a non-complete tripotent in an elementary} and the counterexample in page \pageref{eq counterexample biorthog complement}).

\begin{lemma}\label{l biorthogonal of a non-complete tripotent in a wk JBstar} Let $\displaystyle E = \bigoplus_{i\in I}^{c_0} K_i$ be a weakly compact JB$^*$-triple, where each $K_i$ is an elementary JB$^*$-triple. For each $i\in I$, let $\pi_i$ denote the projection of $E$ onto $K_i$. Suppose $e$ is a tripotent in $E$ such that $\pi_i (e)$ is a non-complete tripotent in $K_i$ for every $i\in I$.  Then the identity $\{e\}^{\perp \perp}_{_E}=E_0(e)^\perp= E_2(e)$ holds.
\end{lemma}

\begin{proof} For each $x\in E$, we shall write $x= (x_i)_{i\in I}$ with $x_i = \pi_i(x)$. Since $e$ is a tripotent in $E$ the set $I_1:=\{ i \in I \ : \ \pi_i(e)\neq 0\}$ must be finite. We observe that $$\{e\}_{_E}^{\perp} = \left\{x\in E \ : \ x_i \in \{e_i\}_{_{K_i}}^{\perp} \hbox{ for all } i \in I_1\right\},$$ and hence $$\{e\}_{_E}^{\perp\perp} = \left\{x\in E \ : \ x_i\in \{e_i\}_{_{K_i}}^{\perp\perp} \hbox{ for all } i \in I_1, \ x_i  =0 \hbox{ for all } i \in I\backslash I_1 \right\}.$$ Since for $i\in I_1$, $e_i$ is a non-complete tripotent in $K_i$ it follows from Lemma \ref{l biorthogonal of a non-complete tripotent in an elementary} that $\{e_i\}_{_{K_i}}^{\perp\perp} = (K_i)_2(e_i)$, and consequently $\{e\}_{_E}^{\perp\perp} =K_2(e)$.
\end{proof}

Let us observe that the conclusion in the previous Lemma \ref{l biorthogonal of a non-complete tripotent in a wk JBstar} remains true when $E$ is replaced with an $\ell_{\infty}$-sum of Cartan factors.

\section{New properties derived from the facial structure of an elementary JB$^*$-triple}\label{sec:new facial properties}

Along this paper, given a Banach space $X$, the symbols $\mathcal{B}_{X}$ and $S(X)$ will stand for the closed unit ball and the unit sphere of $X$, respectively. 
\smallskip

The main goal of this paper is to prove that every weakly compact JB$^*$-triple satisfies the Mazur-Ulam property. In a first step we shall study this property in the case of an elementary JB$^*$-triple $K$. If $K$ is reflexive it follows that $K=K^{**}$ is a Cartan factor, and hence the desired property follows from \cite[Theorem 4.14, Proposition 4.15 and Remark 4.16]{BeCuFerPe2018} when $K$ has rank one or rank bigger than or equal to three and from \cite[Theorem 1.1.]{KalPe2019} in the remaining cases. We shall therefore restrict our study to the case in which $K$ is a non-reflexive elementary JB$^*$-triple (equivalently, an infinite dimensional elementary JB$^*$-triple of type 1, 2 or 3).\label{label restriction to non-reflexive elementary}\smallskip

As in many previous studies, the facial structure of the closed unit ball of a Banach space $X$ is a key tool to determine if $X$ satisfies the Mazur--Ulam property. The main reason being the fact that a surjective isometry $\Delta$ between the unit spheres of two Banach spaces $X$ and $Y$ maps maximal proper faces of $\mathcal{B}_{X}$ to maximal proper faces of $\mathcal{B}_{Y}$ (cf. \cite[Lemma 5.1]{ChenDong2011}, \cite[Lemma 3.5]{Tan2014} and \cite[Lemma 3.3]{Tan2016}).\smallskip

We recall that a convex subset $F$ of a convex set $\mathcal{C}$ is called a \emph{face} of $\mathcal{C}$ if for every $x\in F$ and every $y,z\in \mathcal{C}$ such that $x= t y + (1-t) z$ for some $t\in [0,1]$, we have $y,z\in F$. Let us observe that every proper (i.e., non-empty and non-total) face of the closed unit ball of a Banach space $X$ is contained in $S(X)$. Following the notation in \cite{MoriOza2018}, a closed face $F \subseteq S(X)$ is called an \emph{intersection face} if
$$ F = \bigcap \Big\{ E \ : \ E \subseteq S(X) \hbox{ a maximal face containing } F \Big\}.$$ If $X$ is a complex C$^*$-algebra or the the predual of a von Neumann algebra, or more generally, a JB$^*$-triple or the predual of a JBW$^*$-triple, every proper norm closed face of $\mathcal{B}_{X}$ is an intersection face (see \cite[Corollary 3.4]{Tan2017b} and \cite[Proof of Proposition 2.4 and comments after and before Corollary 2.5]{FerGarPeVill17}). It should be remarked that this conclusion can be also derived from the main results in \cite{EdFerHosPe2010}. These facts together with \cite[Lemma 8]{MoriOza2018} are employed in the next result which was already stated in \cite[Lemma 2.2]{KalPe2019}.

\begin{lemma}\label{l intersection faces MoriOzawa L8}{\rm(\cite[Lemma 2.2]{KalPe2019}, \cite[Lemma 8]{MoriOza2018}, \cite[Proposition 2.4]{FerGarPeVill17}, \cite[Corollary 3.11]{EdFerHosPe2010})} Let $\Delta: S(E)\to S(Y)$ be a surjective isometry where $E$ is a JB$^*$-triple and $Y$ is a real Banach space. Then $\Delta$ maps proper norm closed faces of $\mathcal{B}_{E}$ to intersection faces in $S(Y)$. Furthermore, if $F$ is a proper norm closed face of $\mathcal{B}_{E}$ then $\Delta(-F) = -\Delta(F)$.
\end{lemma}

The structure of all norm closed faces of the closed unit ball of a JB$^*$-triple $E$ was completely determined in \cite{EdFerHosPe2010}, where it is shown that each norm closed face of $\mathcal{B}_{E}$ is univocally given by a tripotent in $E^{**}$ which is compact relative to $E$ (see \cite[Theorem 3.10 and Corollary 3.12]{EdFerHosPe2010} and the concrete definitions therein). As we observed before, each weakly compact JB$^*$-triple $E$ contains all tripotents in $E^{**}$ which are compact relative to $E$ (see \cite[Corollary 2.5]{PeTan16}). Therefore, the norm closed faces of the closed unit ball of a weakly compact JB$^*$-triple $E$ are completely determined by the tripotents in $E$, consequently Theorem 3.10 in \cite{EdFerHosPe2010} in this concrete setting assures that for each proper norm closed face $F$ of $\mathcal{B}_{E}$ there exists a non-zero finite rank tripotent $e\in E$ such that \begin{equation}\label{eq structure of norm closed proper faces in the closed unit ball of weakly compact JB*triples} F= F_e = e + \mathcal{B}_{E_0(e)} = (e+E_0(e))\cap \mathcal{B}_{E}.
\end{equation}

Let us comment some of the difficulties we can find when applying our current knowledge. Suppose $\Delta: S(K(H)) \to S(Y)$ is a surjective isometry, where $H$ is an infinite dimensional complex Hilbert space and $Y$ is a real Banach space. For each non-zero partial isometry $e\in K(H)$, Lemma \ref{l intersection faces MoriOzawa L8} shows that $\Delta (F_e)$ is an intersection face in $S(Y)$ and the restriction $\Delta|_{F_{e}} : F_e \to \Delta (F_e)$ is a surjective isometry too. Henceforth, given an element $x_0$ in a Banach space $X,$ we shall write $\mathcal{T}_{x_0}: X\to X$ for the translation mapping defined by $\mathcal{T}_{x_0} (x) = x+x_0$ ($x\in X$). By considering the commutative diagram
$$\begin{tikzcd}
F_e \arrow{r}{\Delta|_{F_e}} \arrow[swap]{d}{\mathcal{T}_{-e}}
& \Delta({F}_e) \\
(1-ee^*) \mathcal{B}_{ K(H)}  (1-e^*e) \arrow[dashrightarrow, "\Delta_{e}"]{ru} &
\end{tikzcd}$$ we realize that if we could prove that $\mathcal{B}_{(1-ee^*) K(H) (1-e^*e)}$ satisfied the strong Mankiewicz property, we could get some progress to determine the behavior of $\Delta$ on $F_e$. However, $(1-ee^*) K(H) (1-e^*e)$ is a JB$^*$-triple whose closed unit ball contains no extreme points (let us observe that $H$ is infinite dimensional with $ee^*$ and $e^* e$ finite rank projections). So, our current technology is not enough to attack the problem from this perspective. We shall develop a new facial argument not contained in the available literature.\smallskip

The following result is implicit in \cite[Remark 20]{FerMarPe2012} and a detailed explanation can be found in \cite[Lemma 3.3]{PeTan16} from where it has been taken.

\begin{lemma}\label{l two finite rank tripotents at distance 1 are orthogonal}{\cite[Lemma 3.3]{PeTan16}} Let $e$ be a tripotent in a JB$^*$-triple $E$. Suppose $x$ is an element in $\mathcal{B}_E$ satisfying $\|e\pm x\| = 1$. Then $x\perp  e$.
\end{lemma}

We shall need the next consequence.

\begin{lemma}\label{l two finite rank tripotents at distance leq 1 are orthogonal} Let $e$ be a tripotent in a JB$^*$-triple $E$. Suppose $x$ is an element in $S(E)$ satisfying $\|e\pm x\| \leq 1$. Then $x\perp  e$.
\end{lemma}

\begin{proof} Since $S(E)\ni x = \frac12 (x+e )+ \frac12 (x-e)$, we deduce from $\|e\pm x\| \leq 1$ that $\|e\pm x\| = 1$. Lemma \ref{l two finite rank tripotents at distance 1 are orthogonal} implies that $ x\perp e$ as desired.
\end{proof}

The next result has been borrowed from \cite{KalPe2019}.

\begin{lemma}\label{c Tingley antipodal thm for tripotents}\cite[Corollary 2.4]{KalPe2019} Let $\Delta: S(E)\to S(Y)$ be a surjective isometry where $E$ is a JB$^*$-triple and $Y$ is a real Banach space. Suppose $e$ is a non-zero tripotent in $E$, then $\Delta(-e) = -\Delta(e)$.
\end{lemma}

We are now in position to establish a new geometric result, based on the facial structure of $\mathcal{B}_{K(H)}$, which provides a new tool to prove the Mazur--Ulam property in the case of elementary JB$^*$-triples.

\begin{proposition}\label{p new geometric facial property} Let $\Delta: S(K) \to S(Y)$ be a surjective isometry, where $K$ is an elementary JB$^*$-triple and $Y$ is a real Banach space. Then for each tripotent $e\in K$ and each minimal tripotent $u\in K$ with $e\perp u$ the set $\Delta(u+ \mathcal{B}_{K_2(e)})$ is convex and the restriction $\Delta|_{u+ \mathcal{B}_{K_2(e)}}: u+ \mathcal{B}_{K_2(e)} \to \Delta(u + \mathcal{B}_{K_2(e)})$ is an affine mapping. Consequently, there exists a real linear isometry $T^{u}_e$ from $K_2(e)$ onto a norm closed subspace of $Y$ satisfying $\Delta( \mathcal{T}_{u} (x) )=\Delta( u +x) = T^{u}_e (x) + \Delta(u)$ for all $x\in \mathcal{B}_{K_2(e)}.$
\end{proposition}

\begin{proof} As we commented in page \pageref{label restriction to non-reflexive elementary}, by \cite[Theorem 4.14, Proposition 4.15 and Remark 4.16]{BeCuFerPe2018} and \cite[Theorem 1.1.]{KalPe2019}, we can assume that $K$ is non-reflexive (i.e. an infinite dimensional elementary JB$^*$-triple of type 1, 2 or 3), and hence every tripotent in $K$ is non-complete and of finite rank and $K$ has infinite rank.\smallskip

Let $e$ and $u$ be non-zero tripotents in $K$ such that $u$ is minimal and $u\perp e$. The conclusion clearly holds for $e=0$, we can therefore assume that $e$ is a non-zero finite rank tripotent. Set $w= e + u$. Keeping the notation in \eqref{eq structure of norm closed proper faces in the closed unit ball of weakly compact JB*triples} for each non-zero tripotent $v\in K$ we write $F_v =  v + \mathcal{B}_{K_0 (v)}$ for the proper norm closed face of $\mathcal{B}_{K}$ associated with $v$.\smallskip

Let $\mathcal{D}_0 =\Delta(F_{u})$. Lemma \ref{l intersection faces MoriOzawa L8} implies that $\mathcal{D}_0$ is an intersection face in $S(Y)$. Let us fix an arbitrary minimal tripotent $v\in K$ such that $v\perp w$. We set $$\mathcal{D}_1^v := \left\{ y\in \mathcal{D}_0 \ : \  \|y\pm \Delta(v) \|\leq 1 \right\}.$$ We claim that $\mathcal{D}_1^v$ is norm closed and convex. Namely, given $y_1,y_2\in \mathcal{D}_1^v$ and $t\in [0,1]$ the convex combination $t y_1 + (1-t) y_2 \in \mathcal{D}_0$ because $\mathcal{D}_0$ is convex. Furthermore $$ \| t y_1 + (1-t) y_2 \pm \Delta (v) \| \leq t\ \| y_1 \pm \Delta (v) \| +  (1-t)\ \| y_2 \pm \Delta (v) \| \leq 1,$$ witnessing that $t y_1 + (1-t) y_2\in \mathcal{D}_1^v$. Clearly $\mathcal{D}_1^v$ is norm closed.\smallskip

Let us consider the inner ideal $K_2(e)$. Clearly $K_2(e)$ is a finite dimensional JBW$^*$-triple. We shall next prove that \begin{equation}\label{eq image of BM + pm} \Delta\left(\mathcal{B}_{K_2(e)} + u \right)\! =\! \bigcap \Big\{ \mathcal{D}_1^v  : v \hbox{ a minimal tripotent in } K \hbox{ with } v\perp w \Big\}.
\end{equation}

$(\subseteq)$ Let $v\in K$ be a minimal tripotent with $v\perp w$. We take an element $u+z \in u + \mathcal{B}_{K_2(e)}$. Since $u+ z   \perp v,$ $\Delta(-v) = -\Delta(v)$ (see Lemma \ref{c Tingley antipodal thm for tripotents}), and $\Delta$ is an isometry we have $$\left\| \Delta(u +z) \pm \Delta (v) \right\| = \left\| z+u \pm v \right\| =\max\{\left\|z+u\right\|, \|v\| \} =1. $$

$(\supseteq)$ Suppose next that $y\in \mathcal{D}_1^v $ for every minimal tripotent $v$ in $K$ with $v\perp w$. It follows from the definition that $y \in \mathcal{D}_0 =\Delta(F_{u})$ and hence there exists $x\in F_{u}$ satisfying $\Delta(x) =y$. Thus, by Lemma \ref{c Tingley antipodal thm for tripotents} and the assumptions, we have $1\geq \|y\pm \Delta(v) \| = \| x\pm v \|,$ for every $v$ as above. Lemma \ref{l two finite rank tripotents at distance leq 1 are orthogonal} implies that $x\perp v$ for every minimal tripotent $v\in K$ with $v\perp w$. It follows from Corollary \ref{c biorthogonal in an elementary through the orthogonal minimal tripotents} that $x\in K_2 (w)$, and since $x\in F_{u}=  u + \mathcal{B}_{K_0 (u)}$ we can easily see, for example from \eqref{eq Peirce 2 of an orthogonal sum}, that $x\in u+ \mathcal{B}_{K_2(e)}$, as desired.\smallskip

We consider next the following commutative diagram
\begin{equation}\label{eq diagram Delta u e} \begin{tikzcd}  u + \mathcal{B}_{K_2(e)} \arrow[swap]{d}{\tau_{-u}} \arrow{r}{\Delta} &  \Delta\left( u + \mathcal{B}_{K_2(e)} \right) \\
  \mathcal{B}_{K_2(e)}  \arrow[dashrightarrow]{r}{\Delta^{u}_e} &  \Delta\left( u + \mathcal{B}_{K_2(e)} \right)-\Delta(u) \arrow[u, "\tau_{\Delta(u)}"]
\end{tikzcd}
 \end{equation}
It follows from \eqref{eq image of BM + pm} that $\Delta\left( u + \mathcal{B}_{K_2(e)} \right),$ and hence $\Delta\left( u + \mathcal{B}_{K_2(e)} \right)-\Delta(u),$ is a norm closed convex subset of $S(Y)$. Since $K_2(e)$ is a finite dimensional JBW$^*$-triple, it follows from \cite[Corollary 2.2]{BeCuFerPe2018} that $\mathcal{B}_{K_2(e)}$ satisfies the strong Mankiewicz property. Therefore, by the strong Mankiewicz property there exists a surjective real linear isometry $T_e^u $ from $K_2(e)$ onto a norm closed subspace of $Y$ whose restriction to  $\mathcal{B}_{K_2(e)}$ is $\Delta^{u}_e$, that is $$ \Delta (u +x) =\Delta(u) + \Delta^{u}_e (x) = \Delta(u) + T^{u}_e (x),$$ for all $x\in  \mathcal{B}_{K_2(e)}$.
\end{proof}


Let us remark a consequence of the previous Lemma \ref{l two finite rank tripotents at distance leq 1 are orthogonal} and Corollary \ref{c biorthogonal in an elementary through the orthogonal minimal tripotents}. The statement has been actually outlined in the proof of the previous proposition.

\begin{lemma}\label{l sphere M minimal tripotents orhtogonal to e} Let $e$ be a non-complete tripotent in an elementary JB$^*$-triple $K$. Let $\{e\}^{\perp}_{min}=\{v \hbox{ minimal tripotent in } K \hbox{ with } v\perp e\}.$ Then $$\mathcal{B}_{K_2(e)} = \{ x\in \mathcal{B}_{K} : \|x-v\|\leq 1 \hbox{ for all } v\in \{e\}^{\perp}_{min} \}.$$
\end{lemma}

\begin{proof}$(\subseteq)$ Take $x\in \mathcal{B}_{K_2(e)}$. It follows from Peirce arithmetic that $x\perp v$ for all $v\in \{e\}^{\perp}_{min}$, and thus $\|x- v\| = \max\{\|x\|,\|v\|\} \leq 1$.\smallskip

$(\supseteq)$ Assume now that $x\in \mathcal{B}_{K}$ with $\|x-v\|\leq 1$ for all $v\in \{e\}^{\perp}_{min}$. We observe that $-v\in \{e\}^{\perp}_{min}$ for all $v\in \{e\}^{\perp}_{min}$. Therefore $\|x\pm v\|\leq 1$ for all $v\in \{e\}^{\perp}_{min}$. Lemma \ref{l two finite rank tripotents at distance leq 1 are orthogonal} implies that $x\perp v$ for all $v\in \{e\}^{\perp}_{min}$. Corollary \ref{c biorthogonal in an elementary through the orthogonal minimal tripotents} proves that $x\in K_2(e)$ as desired.
\end{proof}

In the following proposition we explore the properties of the real linear isometries $T^{u}_e$ given by Proposition \ref{p new geometric facial property}.

\begin{proposition}\label{p first consequence from the new geometric facial property} Let $\Delta: S(K) \to S(Y)$ be a surjective isometry, where $K$ is an elementary JB$^*$-triple and $Y$ is a real Banach space. Suppose $e$ and $v$ are tripotents in $K$ with $v$ minimal and $v\perp e$. Let $T^{v}_e : K_2(e) \to Y$ be the real linear isometry given by Proposition \ref{p new geometric facial property}. Then the following statements hold:
\begin{enumerate}[$(a)$]\item $T^{v}_e = T^{-v}_{-e} = T^{-v}_e = T^{v}_{-e}$;
\item Suppose $u$ is a minimal tripotent with $u\perp v$. Then $\Delta(u) = T_{u}^{v} (u)$;
\item Suppose $u$ is a minimal tripotent in $K_2(e)$. Then $\Delta(u) = T_{e}^{v} (u)$;
\item For each $x\in S(K_2(e))$ we have $ T^{v}_e (x) = \Delta(x);$
\item If $w$ is another minimal tripotent with $w\perp e$, the real linear isometries $T^{v}_e$ and $T^{w}_e$ coincide;
\end{enumerate}
\end{proposition}

\begin{proof} If $K$ is reflexive, the desired conclusion is an easy consequence of \cite[Theorem 4.14, Proposition 4.15 and Remark 4.16]{BeCuFerPe2018} and \cite[Theorem 1.1.]{KalPe2019} because the latter results show that $\Delta$ admits an extension to a surjective real linear isometry. We shall therefore assume that $K$ is non-reflexive. As we commented before, under our hypotheses, $K_2(e)$ is a (weakly compact) finite dimensional JBW$^*$-triple, and thus every element in $K_2(e)$ can be written as a finite positive combination of mutually orthogonal minimal projections in $K$.\smallskip

$(a)$ It follows from the arguments in the previous paragraph that it suffices to show that $T^{v}_e  (u)= T^{-v}_{-e} (u) = T^{-u}_e (u) = T^{u}_{-e} (u)$ for every minimal tripotent $u$ in $K_2(e)= K_2(-e)$. Fix a minimal tripotent $u\in K_2(e)$. By Lemma \ref{c Tingley antipodal thm for tripotents} and Proposition \ref{p new geometric facial property} we have $$\Delta(v+u) = \Delta(v) + T^{v}_e (u) = - \Delta(-v-u) = \Delta(v) - T^{-v}_e (-u) = \Delta(v) + T^{-v}_e (u),$$
$$\Delta(v-u) = \Delta(v) - T^{v}_e (u) = - \Delta(-v+u) 
 = - \Delta(-v) - T^{-v}_{-e} (u)=  \Delta(v) - T^{-v}_{-e} (u),$$ witnessing the desired equalities.\smallskip

$(b)$ The elements $u\pm v $ belong to the convex set $u + \mathcal{B}_{K_2(v)}$. By Proposition \ref{p new geometric facial property} $\Delta$ is affine on $u + \mathcal{B}_{K_2(v)}$. Therefore $\Delta(u) = \frac12 \Delta(u +v ) +\frac12 \Delta(u-v).$ By applying Proposition \ref{p new geometric facial property}, Lemma \ref{c Tingley antipodal thm for tripotents} and $(a)$ we deduce that $$\Delta(u) = \frac12 \Delta(u +v ) +\frac12 \Delta(u-v) = \frac12 \left(\Delta(v) + T_{u}^v (u) + \Delta(-v) + T_{u}^{-v} (u) \right) = T_{u}^v (u).$$

$(c)$ By Proposition \ref{p new geometric facial property} and statement $(b)$ we get $$\Delta(v) + T_e^v(u) =\Delta (v+u ) = \Delta(v) + T_u^v(u) = \Delta(v) + \Delta (u),$$ which proves the desired identity.\smallskip

$(d)$ Let us take $x \in S(K_2(e))$. Having in mind the arguments at the beginning of this proof, we can find mutually orthogonal minimal tripotents $e_1,\ldots, e_m$ such that $\displaystyle x = e_1 +\sum_{k=2}^m t_k e_k$, where $t_2,\ldots, t_m\in (0,1]$. Set $w= e_2+\ldots+e_m$. Proposition \ref{p new geometric facial property} applied to $w$ and $e_1$ guarantees that $$\Delta(x) = \Delta (e_1) + \sum_{k=2}^m t_k T_w^{e_1} (e_k) = \Delta (e_1) + \sum_{k=2}^m t_k \Delta (e_k)$$ $$ = T_e^{v} (e_1) + \sum_{k=2}^m t_k T_e^{v} (e_k)=T_e^{v}\left( e_1 +\sum_{k=2}^m t_k e_k \right) = T_e^{v} (x),$$ where in the second and third equalities we applied $(c)$.\smallskip

$(e)$ This is a trivial consequence of $(d)$ because by this statement the real linear isometries $T^{v}_e$ and $T^{w}_e$ coincide with $\Delta$ on $S(K_2(e))$.
\end{proof}

We can establish next a version of \cite[Corollary 2.7]{KalPe2019} within the framework of elementary JB$^*$-triples.

\begin{corollary}\label{c affine on convex subsets of the sphere in reflexive}  Let $\Delta: S(K) \to S(Y)$ be a surjective isometry, where $K$ is an elementary JB$^*$-triple and $Y$ is a real Banach space. Let $\mathcal{C}\subset S(K)$ be a non-empty convex subset. Then $\Delta|_{\mathcal
{C}}$ is an affine mapping.
\end{corollary}

\begin{proof} If $K$ is reflexive, then the conclusion follows from \cite[Corollary 2.7]{KalPe2019}. Let us assume that $K$ is non-reflexive. Let $\mathcal{C}\subset S(K)$ be a non-empty convex subset, $x,y\in \mathcal{C}$ and $t\in [0,1]$. By the structure of elementary JB$^*$-triples, for each $0<\varepsilon<\frac12$, there exists a (finite rank) tripotent $e$ in $K$ and elements $x_1,y_1\in S(K_2(e))$ satisfying $\|x-x_1\|, \|y-y_1\|<\varepsilon$ (cf. \cite[Remark 4.6]{BuChu} and the fact that $K$ is an infinite dimensional elementary JB$^*$-triple). An standard argument shows that $$\left\| t x + (1-t) y  -  (t x_1 + (1-t) y_1) \right\| < \varepsilon, $$ and $$\left\| \frac{t x_1 + (1-t) y_1}{\|t x_1 + (1-t) y_1\|} -  (t x_1 + (1-t) y_1) \right\| < \varepsilon.$$

Since $K$ has infinite rank, we can find a minimal tripotent $v$ in $K$ which is orthogonal to $e$. Let $T_e^v: K_2(e) \to Y$ be the real linear isometry given by Proposition \ref{p new geometric facial property}. Having in mind that $x_1,y_1\in K_2(e)$, it follows from the just quoted proposition, the hypothesis on $\Delta$ and Proposition \ref{p first consequence from the new geometric facial property}$(d)$ that
$$\begin{aligned} & \left\| t\Delta(x) + (1-t) \Delta(y) - \Delta (t x + (1-t) y) \right\| \\
&\leq \left\| t\Delta(x) + (1-t) \Delta(y) - t\Delta(x_1) - (1-t) \Delta(y_1) \right\| \\
& + \left\| tT_e^v (x_1) + (1-t) T_e^v (y_1) - T_e^v (t x_1 + (1-t) y_1) \right\| \\
& + \left\|T_e^v (t x_1 + (1-t) y_1) - T_e^v \left( \frac{t x_1 + (1-t) y_1}{\|t x_1 + (1-t) y_1\|} \right) \right\| \\
& + \left\|\Delta \left( \frac{t x_1 + (1-t) y_1}{\|t x_1 + (1-t) y_1\|} \right) - \Delta (t x + (1-t) y) \right\| <4 \varepsilon.
\end{aligned}$$ The arbitrariness of $0<\varepsilon < \frac12$ implies that $t\Delta(x) + (1-t) \Delta(y) = \Delta (t x + (1-t) y)$ as desired.
\end{proof}

\section{Elementary JB$^*$-triples satisfy the Mazur--Ulam property}\label{sec: elementary JB* triples satisfy MUP}

Our goal in this section is to prove that every weakly compact JB$^*$-triple satisfies the Mazur--Ulam property. By this result we shall exhibit an example of a non-unital and non-commutative C$^*$-algebra containing no unitaries but satisfying the Mazur--Ulam property. The new geometric properties related to the facial structure of elementary JB$^*$-triples developed in the previous section will be the germ to prove our result. Compared with the techniques in \cite{MoriOza2018,BeCuFerPe2018,CuePer19} and \cite{KalPe2019}, we shall not base our proof on \cite[Lemma 6]{MoriOza2018} nor on \cite[Lemma 2.1]{FangWang06}. 

\begin{theorem}\label{t elementary JBstar triples satisfy the MUP} Every elementary JB$^*$-triple $K$ satisfy the Mazur--Ulam property, that is, for every real Banach space $Y$, every surjective isometry $\Delta: S(K)\to S(Y)$ admits a (unique) extension to a surjective real linear isometry from $K$ onto $Y$.
\end{theorem}

\begin{proof} As we already commented in previous sections, if $K$ is reflexive the conclusion follows from \cite[Theorem 4.14, Proposition 4.15 and Remark 4.16]{BeCuFerPe2018} and \cite[Theorem 1.1 and Corollary 1.2]{KalPe2019}. As in the proof of Proposition \ref{p new geometric facial property} we shall assume that $K$ is non-reflexive (i.e. an infinite dimensional elementary JB$^*$-triple of type 1, 2 or 3), and hence every tripotent in $K$ is non-complete and of finite rank and $K$ has infinite rank.\smallskip

Let $\mathcal{F}=\mathcal{F}(K)$ denote the linear subspace of $K$ generated by all minimal tripotents in $K$. In our case $\mathcal{F}$ is a normed subspace of $K$ which is norm dense but non-closed. We shall define a mapping $T: \mathcal{F}\to Y.$ By definition every element $x$ in $\mathcal{F}$ can be written as a finite positive combination of mutually orthogonal minimal tripotents in $K$. We can therefore find a tripotent $e$ in $K$ such that $x\in K_2(e)$. Since $K$ has infinite rank we can always find a minimal tripotent $v\in K$ with $v\perp e$. We set $T(x) := T_{e}^v(x)$. Clearly $T(0) = 0 $.\smallskip

We claim that $T$ is well defined. Pick $0\neq x\in \mathcal{F}$. Suppose $e_1,e_2$ are tripotents in $K$ with $x\in K_2(e_j)$ for $j=1,2$ and
$v_1,v_2$ are two minimal tripotents in $K$ with $v_j \perp e_j$ for $j=1,2$. Proposition \ref{p first consequence from the new geometric facial property}$(d)$ implies that $$T_{e_1}^{v_1} \left(\frac{x}{\|x\|}\right)= \Delta\left(\frac{x}{\|x\|}\right) = T_{e_2}^{v_2} \left(\frac{x}{\|x\|}\right),$$ and hence $T$ is well defined.\smallskip

We shall next show that $T$ is linear. For each $x\in \mathcal{F}$, each pair of tripotents $e,v$ in $K$ with $v$ minimal, $v\perp e$ and $x\in K_2(e)$, and each real number $\alpha$, we have $\alpha x\in K_2(e)$ and $T(\alpha x) = T_e^v (\alpha x) = \alpha T_e^v(x) = \alpha T(x)$. Let us now take $x,y \in \mathcal{F}$. We can find a tripotent $e\in K$ such that $x,y, x+y\in K_2(e)$. Let  $v\in K$ be a minimal tripotent with $v\perp e$. By definition $$T(x+y ) = T_e^v (x+y) =  T_e^v (x) +  T_e^v (y) = T(x) + T(y).$$

Since $T: \mathcal{F}\to Y$ is a real linear isometry and $\mathcal{F}$ is norm dense in $K$, we can find a unique extension of $T$ to a surjective real linear isometry from $K$ to $Y$ which will be denoted by the same symbol $T$.\smallskip

We shall finally show that $T$ coincides with $\Delta$ on $S(K)$. By continuity and norm density of $\mathcal{F}$, it suffices to prove that $T$ coincides with $\Delta$ on $S(\mathcal{F})$. By definition, given $x\in S(\mathcal{F})$, tripotents $e,v$ in $K$ with $v\perp e$ and $x\in K_2(e)$, Proposition \ref{p first consequence from the new geometric facial property}$(d)$ assures that $T(x) = T_e^v (x) = \Delta(x)$, which concludes the proof.
\end{proof}

The main results in \cite{FerPe18Adv} prove that every surjective isometry between the unit spheres of two elementary JB$^*$-triples $K_1$ and $K_2$ can be extended to a surjective real linear isometry between $K_1$ and $K_2$ (see \cite[Theorems 4.4-4.7 and 4.9]{FerPe18Adv} and some precedents in \cite{PeTan16}). Our previous Theorem \ref{t elementary JBstar triples satisfy the MUP} offers an strengthened conclusion by showing that every surjective isometry from the unit sphere of an elementary JB$^*$-triple onto the unit sphere of any real Banach space extends to a surjective real linear isometry.\smallskip

The following straightforward consequence is interesting by itself.

\begin{corollary}\label{c KH satisfies the MUP} For each complex Hilbert space $H$, the space $K(H)$ of all compact linear operators on $H$ satisfies the Mazur--Ulam property.
\end{corollary}

\section{More on the strong Mankiewicz property}

In this section we pursue a result showing that every weakly compact JB$^*$-triple satisfies the strong Mankiewicz property. We begin by establishing a technical result which is valid for an abstract class of Banach spaces including all weakly compact JB$^*$-triples.

\begin{proposition}\label{p strong Mankiewicz property for c0 sums} Let $(X_{i})_{i\in I}$ be a family of Banach spaces, and let $X= \displaystyle \bigoplus_{i\in I}^{c_{0}} X_i$. Suppose that the following hypotheses hold: for each $x,y\in X$ and each $\varepsilon>0$ there exist a finite subset $F\subseteq I$ {\rm(}depending on $x,y$ and $\varepsilon>0${\rm)}, closed subspaces $Z_i\subseteq X_i$ and elements $a_i,b_i\in Z_i$ {\rm(}$i\in F${\rm)} such that
\begin{enumerate}[$(1)$]\item For each $i\in F$ there exists a subset $M_i\subseteq \mathcal{B}_{X_i}$ satisfying $$\mathcal{B}_{Z_i} = \left\{ x_i \in \mathcal{B}_{X_i} : \|x_i -m_i\|\leq 1 \hbox{ for all } m_i \in M_i \right\};$$
\item $\|x - (a_i)_{i\in F}\|, \|y - (b_i)_{i\in F}\| <\varepsilon$ {\rm\Big(}we obviously regard $\displaystyle \bigoplus_{i\in F}^{\ell_{\infty}} Z_i$ as a closed subspace of $X${\rm\Big)};
\item The closed unit ball of the space $\displaystyle \bigoplus_{i\in F}^{\ell_{\infty}} Z_i$ satisfies the strong Mankiewicz property.
 \end{enumerate}  Then  every convex body $K\subset X$ satisfies the strong Mankiewicz property.
\end{proposition}

\begin{proof} By \cite[Lemma 4]{MoriOza2018} it suffices to prove that the closed unit ball of $X$ satisfies the strong Mankiewicz property. To this end let $\Delta :\mathcal{B}_{X}\to L$ be a surjective isometry, where $L$ is a convex subset in a normed space $Y$. We shall show that $\Delta$ is affine.\smallskip

Let us take $x,y\in \mathcal{B}_{X}$ and $t\in (0,1)$. Fix an arbitrary  $\varepsilon >0.$ It follows from our hypotheses that there exist a finite set $F\subseteq I$, closed subspaces $Z_i\subseteq X_i$ and elements $a_i,b_i\in \mathcal{B}_{Z_i}$ ($i\in F$) such that the closed unit ball of the space $\displaystyle Z= \bigoplus_{i\in F}^{\ell_{\infty}} Z_i$ satisfies the strong Mankiewicz property,
$\|x - (a_i)_{i\in F}\| <\varepsilon$ and $\|y - (b_i)_{i\in F}\| <\varepsilon$.\smallskip

We also know the existence of subsets $M_i\subseteq \mathcal{B}_{X_i}$ ($i\in I$) satisfying $(1)$. Let $M\subseteq X$ denote the set given by $$M:= \left\{ x=(x_i)_i\in \mathcal{B}_X : x_i\in M_i \hbox{ for all } i \in F,\ x_j\in \mathcal{B}_{X_j} \hbox{ for } j\in I\backslash F\right\}.$$ We also set $$L_1 :=\left\{ y\in L : \|y - \Delta(m)\|\leq 1 \hbox{ for all } m\in M \right\}.$$ Having in mind that $L$ is convex, it is easy to see that $L_1$ also is convex.\smallskip

We claim that \begin{equation}\label{eq Delta BZ equals L1} \Delta(\mathcal{B}_{Z}) = L_1.
\end{equation} 

Indeed, by the assumptions each $z\in \mathcal{B}_{Z}$ satisfies that $\|z-m\|\leq 1$ for all $m\in M$, and hence $\|\Delta(z) -\Delta (m)\|\leq 1$ for all $m\in M$. This proves that $\Delta(\mathcal{B}_{Z}) \subseteq  L_1$.\smallskip

Take now, $y\in L_1$. Since $\Delta$ is surjective there exists (a unique) $z=(z_i)_i\in \mathcal{B}_{X}$ with $\Delta(z) = y.$ Since $\Delta$ is an isometry and $y\in L_1$ we deduce that $$\|z-m\|=\|\Delta(z) -\Delta (m)\|\leq 1,\hbox{ for all } m\in M,$$ in particular, $$\hbox{$\|z_i -m_i\|\leq 1,$ for all $m_i\in M_i$ and for all $i\in F$}$$ and $$\hbox{$\|z_i - w_i\|\leq 1,$ for all $w_i\in \mathcal{B}_{X_i}$ and all $i\in I\backslash F$.}$$ It follows from the hypothesis $(1)$ that $z_i\in \mathcal{B}_{Z_i}$ for all $i\in F,$ and clearly $z_i =0$ for all $i\in I\backslash F$. Therefore $z\in \mathcal{B}_{Z}$ which finishes the proof of \eqref{eq Delta BZ equals L1}.\smallskip

We deduce from \eqref{eq Delta BZ equals L1} and the preceding comments that $\Delta|_{\mathcal{B}_{Z}}:\mathcal{B}_{Z}\to L_1$ is a surjective isometry and $L_1$ is a convex subset of $Y$. We apply now that the closed unit ball of the space $Z$ satisfies the strong Mankiewicz property to deduce that $\Delta|_{\mathcal{B}_{Z}}$ is affine, and thus $$ \Delta \left(t (a_i)_{i\in F}+ (1-t) (b_i)_{i\in F} \right) = t \Delta \left( (a_i)_{i\in F} \right) +  (1-t)  \Delta \left((b_i)_{i\in F} \right),$$ and consequently $$\begin{aligned}\|& \Delta (t x + (1-t) y) - (t \Delta (x) + (1-t) \Delta(y)) \| \\ & \leq \|\Delta (t x + (1-t) y) - \Delta \left(t (a_i)_{i\in F}+ (1-t) (b_i)_{i\in F} \right) \| \\
&+ t  \| \Delta \left( (a_i)_{i\in F}\right)- \Delta (x)  \| + (1-t) \| \Delta \left((b_i)_{i\in F} \right) - \Delta(y) \| < 2\varepsilon.
\end{aligned}$$

The arbitrariness of $\varepsilon>0$ implies that $\Delta (t x + (1-t) y) = t \Delta (x) + (1-t) \Delta(y)$, which concludes the proof.
\end{proof}

An appropriate framework to apply the previous proposition is provided by weakly compact JB$^*$-triples.

\begin{corollary}\label{c weakly compact JB* triples satisfy the SMP} Every weakly compact JB$^*$-triple $E$ satisfies the hypotheses of the previous Proposition \ref{p strong Mankiewicz property for c0 sums}. Consequently every convex body $K\subset E$ satisfies the strong Mankiewicz property.
\end{corollary}

\begin{proof} As we commented in subsection \ref{subsec:background} every weakly compact JB$^*$-triple coincides with a $c_0$-sum of a family $\{K_i: i\in I\}$ of elementary JB$^*$-triples (cf. \cite{BuChu}). Given $x,y\in E$ and $\varepsilon>0$ there exist a finite subset $F\subseteq I$ satisfying $\| x - (x_i)_{i\in F}\|, \| y - (y_i)_{i\in F}\|< \frac{\varepsilon}{2}$.\smallskip

Fix an arbitrary $i\in F$. If $K_i$ is reflexive (i.e. it is finite dimensional or an elementary JB$^*$-triple of type 4), we know that $K_i$ is a JBW$^*$-triple. In this case we take $Z_i = K_i$, $a_i = x_i,$ $b_i = y_i,$ and $M_i = \{0\}$. Clearly, $$\mathcal{B}_{K_i} = \{ x_i \in K_i : \| x_i -0 \|\leq 1\}.$$

Suppose next that $K_i$ is not reflexive (i.e. an infinite dimensional elementary JB$^*$-triple of type 1, 2 or 3). By \cite[Remark 4.6]{BuChu} and/or basic theory of compact operators we can find a finite rank (and hence non-complete) tripotent $e_i$ (i.e. a tripotent which is a finite sum of mutually orthogonal minimal tripotents in $K_i$) such that $\|x_i -P_2(e_i) (x_i)\|,\|y_i -P_2(e_i) (y_i)\| <\frac{\varepsilon}{2}$. The subtriple $Z_i = (K_i)_2 (e_i)$ is finite dimensional and thus a JBW$^*$-triple. Take $a_i = P_2(e_i) (x_i),$ $b_i = P_2(e_i) (y_i)\in Z_i$. Set $M_i:=\{e_i\}_{min}^{\perp} \cap K_i =\{ v \hbox{ minimal tripotent in } K_i : v\perp e_i\}\subseteq \mathcal{B}_{K_i}$. Lemma \ref{l sphere M minimal tripotents orhtogonal to e} assures that $$\mathcal{B}_{Z_i}= \mathcal{B}_{ (K_i)_2 (e_i)} = \left\{ x_i \in \mathcal{B}_{K_i} : \|x_i -m_i\|\leq 1 \hbox{ for all } m_i \in M_i= \{e_i\}_{min}^{\perp} \cap K_i \right\}.$$

\noindent By construction, $$\|x-(a_i)_{i\in F}\|\leq  \| x - (x_i)_{i\in F}\| +  \| (x_i)_{i\in F}- (a_i)_{i\in F}\|< \varepsilon.$$

Finally, for each $i\in F$, $Z_i$ is finite dimensional or reflexive, therefore $\displaystyle \bigoplus_{i\in F}^{\ell_{\infty}} Z_i$ is a JBW$^*$-triple, and thus its closed unit ball satisfies the strong Mankiewicz property by \cite[Corollary 1.2]{KalPe2019}.
\end{proof}

\begin{remark}\label{remark compact C* algebras satisfy strong Mankiewicz property} It is worth to note that every compact C$^*$-algebra of the form $\displaystyle A= \bigoplus_{i\in I}^{c_0} K(H_i),$ where the $H_i$'s are complex Hilbert spaces {\rm(}cf. \cite[Theorem 8.2]{Alex}{\rm)}. Obviously, $A$ is a weakly compact JB$^*$-triple, and hence every convex body $K\subset A$ satisfies the strong Mankiewicz property.
\end{remark}


We have already developed enough tools to address the question whether every weakly compact JB$^*$-triple satisfies the Mazur--Ulam property.

\begin{theorem}\label{t weakly compact JB*-triples satisfy MUP} Every weakly compact JB$^*$-triple $E$ satisfies the Mazur--Ulam property, that is,  for every real Banach space $Y$, every surjective isometry from $ S(E)$ onto $S(Y)$ admits a (unique) extension to a surjective real linear isometry from $E$ onto $Y$.
\end{theorem}

\begin{proof} We begin with an observation. For each non-zero tripotent $u\in E,$ we consider the proper face $F_u = u + \mathcal{B}_{E_0(u)}$. Lemma \ref{l intersection faces MoriOzawa L8} assures that $\Delta(F_u)$ is an intersection face of $S(Y)$, in particular a convex subset. Let us observe that $E_0(u)$ is a weakly compact JB$^*$-triple. Corollary \ref{c weakly compact JB* triples satisfy the SMP} implies that $\mathcal{B}_{E_0(u)}$ satisfies the strong Mankiewicz property. We consider the next commutative diagram
$$ \begin{tikzcd}  F_u =u + \mathcal{B}_{E_0(u)} \arrow[swap]{d}{\tau_{-u}} \arrow{r}{\Delta} &  \Delta\left( u + \mathcal{B}_{E_0(u)} \right) \\
  \mathcal{B}_{E_0(u)}  \arrow[dashrightarrow]{r}{\Delta_{u}} &  \Delta\left( u + \mathcal{B}_{E_0(u)} \right)-\Delta(u) \arrow[u, "\tau_{\Delta(u)}"]
\end{tikzcd}
$$
Since $E_0(u)$ satisfies the strong Mankiewicz property and $\Delta_{u}$ is a surjective isometry from $\mathcal{B}_{E_0(u)} $ onto a convex set, we can guarantee the existence of a linear isometry $T_u: E_0(u) \to Y$ satisfying
\begin{equation}\label{eq definition of Tu i nlast theorem} \Delta (u + x) = T_u (x) + \Delta (u), \hbox{ for all } x\in \mathcal{B}_{E_0(u)},
\end{equation} (cf. \cite{MoriOza2018}). In particular $\Delta|_{F_u}$ is affine.\smallskip

We claim that \begin{equation}\label{eq main property of Tu in last theorem} \Delta (w) = T_u (w), \hbox{ for every non-zero tripotents } u,w \in E \hbox{ with } w\in {E_0(u)} .
\end{equation} 
Namely, for each non-zero tripotent $w\in {E_0(u)}$ ($u\perp w$) the element $u \pm w$ is a tripotent in $F_u$. Therefore $$\Delta (u) = \Delta\left(\frac12 (u+w) + \frac12 (u-w) \right) = \frac12 \Delta\left( u+w\right) + \frac12 \Delta\left(u-w \right).$$
By Lemma \ref{c Tingley antipodal thm for tripotents} (se also \cite[Corollary 2.4]{KalPe2019}) we have $- \Delta\left(u-w \right) = \Delta\left(- u+w) \right)$, and by similar arguments to those given above, but now with respect to the face $F_w$, we have $$\Delta (w) = \Delta\left(\frac12 (u+w) + \frac12 (-u+w) \right) = \frac12 \Delta\left( u+w\right) + \frac12 \Delta\left(-u+w \right)$$ $$= \frac12 \Delta\left( u+w\right) - \frac12 \Delta\left(u-w \right).$$ It then follows that \begin{equation}\label{eq Delta on the sum of orthogonal tripotents last theorem} \Delta (u) + \Delta (w) = \Delta (u + w).
\end{equation} The claim in \eqref{eq main property of Tu in last theorem} is a straight consequence of \eqref{eq Delta on the sum of orthogonal tripotents last theorem} and \eqref{eq definition of Tu i nlast theorem}.\smallskip

We continue with our argument. If $E$ is an elementary JB$^*$-triple the result follows from Theorem \ref{t elementary JBstar triples satisfy the MUP}. We can therefore assume that $E$ decomposes as the orthogonal sum of two non-zero weakly compact JB$^*$-triples $A$ and $B$. Let us pick two non-zero tripotents $u_1\in A$ and $u_2\in B$ and the corresponding linear isometries $T_{u_j}: E_0(u_j) \to Y$ given by \eqref{eq definition of Tu i nlast theorem}.\smallskip

Let us observe that $A\subseteq E_0(u_2),$ $B\subseteq E_0(e_1)$ and $E= A\oplus^{\ell_{\infty}} B$. Therefore the mapping $T: E\to Y$, $T(a+b) = T_{u_2} (a) + T_{u_1} (b)$ is a well-defined linear operator. We shall next show that \begin{equation}\label{eq T coincides with Delta final thm}
T(x) = \Delta(x), \hbox{ for all } x\in S(E).
\end{equation}

Let us fix $x\in S(E)$. By Remark 4.6 in \cite{BuChu} there exists a possible finite at most countable family $\{e_n\}$ of mutually orthogonal minimal tripotents in $E$ and $(\lambda_n)\subseteq \mathbb{R}^+$ such that $\lambda_1 =1$ and $\displaystyle x = \sum_{n\geq 1} \lambda_n e_n$. Each $e_n$ lies in $A$ or in $B$, and hence it follows from the definition of $T$ that $T(e_n) =  T_{u_2} (e_n)$ if $e_n\in A$ and $T(e_n) =  T_{u_1} (e_n)$ if $e_n\in B.$ We deduce from \eqref{eq main property of Tu in last theorem} that in any case we have $\Delta (e_n) = T (e_n)$ for all $n\geq 1$. Now we regard $x$ as an element in the face $F_{e_1}$ and we apply the properties of $T_{e_1}$ and \eqref{eq main property of Tu in last theorem} to deduce that $$\Delta (x) = \Delta \left(e_1 + \sum_{n\geq 2} \lambda_n e_n\right) = \Delta(e_1) + \sum_{n\geq 2} \lambda_n T_{e_1} (e_n) =  \Delta(e_1) + \sum_{n\geq 2} \lambda_n \Delta (e_n)$$ $$ =  T(e_1) + \sum_{n\geq 2} \lambda_n T (e_n) = T \left(e_1 + \sum_{n\geq 2} \lambda_n e_n\right) = T(x),$$ witnessing the validity of \eqref{eq T coincides with Delta final thm}.\smallskip

Finally, since $T$ is linear we derive from \eqref{eq T coincides with Delta final thm} and the hypothesis on $\Delta$ that $T$ is a surjective real linear isometry.
\end{proof}

We have previously mentioned that compact C$^*$-algebras are examples of weakly compact JB$^*$-triples. The next corollary perhaps deserves its own place.

\begin{corollary}\label{c compact C*-algebras satisfy MUP final} Suppose $\displaystyle A = \bigoplus_{i\in I}^{c_0} K(H_i)$ is a compact C$^*$-algebra, where each $H_i$ is a complex Hilbert space. Then every surjective isometry from the unit sphere of $A$ onto the unit sphere of any other real Banach space $Y$ admits a (unique) extension to a surjective real linear isometry from $A$ onto $Y$.
\end{corollary}

\medskip\medskip

\textbf{Acknowledgements} Author partially supported by the Spanish Ministry of Science, Innovation and Universities (MICINN) and European Regional Development Fund project no. PGC2018-093332-B-I00, Programa Operativo FEDER 2014-2020 and Consejer{\'i}a de Econom{\'i}a y Conocimiento de la Junta de Andaluc{\'i}a grant number A-FQM-242-UGR18, and Junta de Andaluc\'{\i}a grant FQM375.
\end{document}